\documentclass[10pt,a4paper,reqno]{amsart}
\usepackage{a4wide}
\usepackage{amssymb, amsmath, amsthm,amsfonts}
\usepackage[utf8]{inputenc}
\usepackage{tikz}
\usepackage{enumitem}
\usepackage[norelsize, boxruled]{algorithm2e}
\usepackage{mathtools}

\newcommand{\ZZ}{\mathbb{Z}}
\newcommand{\QQ}{\mathbb{Q}}
\newcommand{\V}{\mathrm{V}}
\newcommand{\U}{\mathrm{U}}
\newcommand{\cc}[1]{\ensuremath{{\texttt{#1}}}}

\definecolor{LinkColor}{rgb}{0,0,0} 
 
\usepackage[bookmarks=false,colorlinks=true,linkcolor=black,citecolor=black,filecolor=black,urlcolor=black]{hyperref} 
\usepackage[capitalise]{cleveref} 

\theoremstyle{plain}
\newtheorem{theorem}{Theorem}[section]

\newtheorem{corollary}[theorem]{Corollary}
\newtheorem{proposition}[theorem]{Proposition}

\theoremstyle{definition}
\newtheorem{definition}[theorem]{Definition}

\newtheorem*{zc*}{Zassenhaus Conjecture (ZC)}
\newtheorem*{pq*}{Prime Graph Question (PQ)}

\theoremstyle{remark}

\newtheorem{remark}[theorem]{Remark}
\newtheorem{example}[theorem]{Example}

\begin{document}

\title[HeLP-package]{HeLP \\ A \textsf{GAP}-package for torsion units in integral group rings}
\author{Andreas B\"achle}
\address{Vakgroep Wiskunde, Vrije Universiteit Brussel, Pleinlaan 2, 1050 Brussels, Belgium}
\email{\href{mailto:abachle@vub.ac.be}{abachle@vub.ac.be}}
\author{Leo Margolis}
\address{Fachbereich Mathematik, Universit\"{a}t Stuttgart, Pfaffenwaldring 57, 70569 Stuttgart, Germany}
\email{\href{mailto:leo.margolis@mathematik.uni-stuttgart.de}{leo.margolis@mathematik.uni-stuttgart.de}}
\date{\today}
\thanks{The first author is supported by the Research Foundation Flanders (FWO - Vlaanderen). \\ \indent This project was partly supported by the DFG priority program SPP 1489.}

\begin{abstract} We briefly summarize the background of the HeLP-method for torsion units in group rings and present some functionality of a \textsf{GAP}-package implementing it. \end{abstract}

\maketitle

\section{The Zassenhaus Conjecture and the Prime Graph Question} 

Considering the integral group ring $\ZZ G$ of a finite group $G$, one question that strikes the eye is: ``How does the unit group $\U(\ZZ G)$ look like?'', in particular which are the \emph{torsion units}, i.e.\ the units of finite order. Clearly, there are the so-called \emph{trivial units} $\pm g$ for $g \in G$. Already in G.~Higman's PhD-thesis \cite{HigmanThesis} it was proved that all the torsion units are of this form, provided $G$ is abelian. As we are not interested in the torsion coming solely from the ring, but rather in the torsion coming from the group-ring-interplay, we consider the group of \emph{normalized units} $\V(\ZZ G)$, i.e.\ the units mapping to $1$ under the augmentation homomorphism \[ \varepsilon \colon \ZZ G \to \ZZ \colon \sum_{g \in G} u_g g \mapsto \sum_{g \in G} u_g. \]
Then $\U(\ZZ G) = \pm \V(\ZZ G)$.

In the non-commutative case, there are in general of course more torsion units than the trivial ones, e.g.\ conjugates of group elements by units of $\QQ G$ which end up in $\ZZ G$ again.  H.J.~Zassenhaus conjectured more than 40 years ago that these are all the torsion units:

\begin{zc*}[\cite{Zassenhaus}] Let $G$ be a finite group and $u$ a torsion unit in $\V(\ZZ G)$. Then there exists a unit $x$ in $\mathbb{Q}G$ such that $x^{-1}ux = g$ for some $g \in G$.
\end{zc*}
\noindent Elements $u, v \in \ZZ G$ which are conjugate by a unit $x \in \QQ G$ are called \emph{rationally conjugate}, denoted by $u \sim_{\QQ G} v$. The Zassenhaus Conjecture is nowadays one of the main open questions in the area of integral group rings. A highlight was certainly Weiss' proof of this conjecture for nilpotent groups \cite{Weiss88,Weiss91}.

As a first step towards the Zassenhaus Conjecture W. Kimmerle formulated a much weaker version of it which has found quite some attention since:

\begin{pq*}[\cite{KimmiPQ}] Let $G$ be a finite group. Let $p$ and $q$ be different primes such that $\mathrm{V}(\mathbb{Z}G)$ contains an element of order $pq$. Does then $G$ posses an element of order $pq$?\end{pq*}
In other words this question asks whether $G$ and $\mathrm{V}(\mathbb{Z}G)$ have the same prime graph.

A method to attack these questions, nowadays known as the HeLP-method, was introduced by Luthar and Passi in \cite{LP89} and later extended by Hertweck in \cite{HertweckBrauer}. The name \emph{HeLP} (\emph{He}rtweck\emph{L}uthar\emph{P}assi) is due to Alexander Konovalov. The method can be applied algorithmically to a concrete group or, if one has generic characters at hand, a series of groups. 

The present note presents a \textsf{GAP}-package implementing this method \cite{HeLP}. There are two main motivations for this program: Making the method available to researchers working in the field and not having an own implementation and giving readers of papers using the method the opportunity to check the results obtained by the method. We describe the method in \cref{section:help-constraints} and discuss several aspects of our implementation in \cref{section:wagner-test,section:implementation}. 

\section{The HeLP-Constraints}\label{section:help-constraints}

Let $G$ always be a finite group.

The possible orders of torsion units in $\ZZ G$ are restricted by the following proposition.

\begin{proposition}\label{prop:orders} Let $u \in \V(\ZZ G)$ be a torsion unit.
\begin{enumerate}[label=\alph*)]
\item The order of $u$ divides the exponent of $G$. \hfill Cohn, Livingstone {\cite[Corollary 4.1]{CL}}
\item If $G$ is solvable, then the order of $u$ coincides with the order of an element of $G$. \\ \phantom{v} \hfill Hertweck {\cite[Theorem]{HertweckSolvable}} 
\end{enumerate} 
\end{proposition}

\begin{definition} Let $u = \sum_{g \in G} u_g g \in \ZZ G$, $x \in G$ and denote by $x^G$ its conjugacy class. Then \[ \varepsilon_x(u) = \sum_{g \in x^G} u_g \] is called the \emph{partial augmentation of $u$ with respect to $x$} (or rather the conjugacy class of $x$). \end{definition}

The following proposition connects (ZC) to partial augmentations.

\begin{proposition}[Marciniak, Ritter, Sehgal, Weiss {\cite[Theorem 2.5]{MRSW}}] Let $u, v \in \V(\ZZ G)$ be torsion units of order $k$. Then $u \sim_{\QQ G} v$ if and only if $\varepsilon_x(u^d) = \varepsilon_x(v^d)$ for all divisors $d$ of $k$ and all $x \in G$. Moreover, $u$ is rationally conjugate to a group element if and only if $\varepsilon_x(u^d) \geq 0$ for all divisors $d$ of $k$ and all $x \in G$. \end{proposition}

Certain partial augmentations vanish a priori:

\begin{proposition}\label{prop:pas_0} Let $u \in \V(\ZZ G)$ be a torsion unit and $x \in G$.
\begin{enumerate}[label=\alph*)]
 \item \label{prop:BH} If $o(u) \not= 1$, then $\varepsilon_1(u) = 0$. \hfill Berman, Higman \cite[Proposition (1.4)]{SehgalBook2}
 \item \label{prop:Hert_pas0} If $o(x) \nmid o(u)$, then $\varepsilon_x(u) = 0$. \hfill Hertweck \cite[Theorem 2.3]{HertweckBrauer}
\end{enumerate}
\end{proposition}

Let $\psi$ be a character of the group $G$. A representation afforded by $\psi$ can be extended linearly to a representation of $\QQ G$ and then restricted to a representation $D$ of the group of units $\U(\QQ G)$. We will denote its character also by $\psi$. Now consider for a torsion unit $u \in \V(\ZZ G)$ of order $k$ a linear character $\chi \colon C_k \simeq \langle u \rangle \to \mathbb{C}$ given by $\chi(u) = \zeta^\ell$, with $\zeta \in \mathbb{C}^\times$ a primitive $k$-th root of unity. Then we have that the multiplicity of $\zeta^\ell$ of $D(u)$ is given by \[ \langle \chi, \psi \rangle_{\langle u \rangle} \in \ZZ_{\geq 0} ,\] where $\langle -, = \rangle_{\langle u \rangle}$ denotes the inner product on the class functions of $C_k \simeq \langle u \rangle$. Working out a explicit formula for this, one obtains part a) of the following proposition. 

 \begin{proposition}\label{prop:HeLP-constraints} Let $G$ be a finite group and $u \in \U(\ZZ G)$ a torsion unit of order $k$. Let $\zeta \in \mathbb{C}^\times$ be a primitive $k$-th root of unity. 
 \begin{enumerate}[label=\alph*)]
  \item Let $\chi$ be an ordinary character and let $D$ be a representation afforded by $\chi$. Then the multiplicity of $\zeta^\ell$ as an eigenvalue of $D(u)$ is given by
   \begin{equation} \mu_\ell(u, \chi) = \frac{1}{k}\sum_{d \mid k} \operatorname{Tr}_{\QQ(\zeta^d)/\QQ}(\chi(u^d)\zeta^{-d \ell}). \tag*{\emph{Luthar, Passi} \cite[Theorem 1]{LP89}} \end{equation} 
   \item Let $p$ be a prime not dividing $k$ and fix an isomorphism $\zeta \mapsto \overline{\zeta}$ between the group of $k$-th roots of unity in characteristic $0$ and those in characteristic $p$.
   Let $\varphi$ be a $p$-Brauer character and $P$ be a representation afforded by $\varphi$. Then the multiplicity of $\overline{\zeta}^\ell$ as an eigenvalue of $P(u)$ is given by
   \begin{equation} \mu_\ell(u, \varphi) = \frac{1}{k}\sum_{d \mid k} \operatorname{Tr}_{\QQ(\zeta^d)/\QQ}(\varphi(u^d)\zeta^{-d \ell}). \tag*{\emph{Hertweck} \cite[Section 4]{HertweckBrauer}} \end{equation} 
  \end{enumerate}
 \end{proposition}

 This proposition is the linchpin of the HeLP-method. Let $u \in \V(\ZZ G)$ be again a torsion unit of order $k$.  For an ordinary character $\chi$ we have $\chi(u) = \sum_{x^G} \varepsilon_x(u) \chi(x)$.  By \cite[Theorem~3.2]{HertweckBrauer} the analogue statement holds for $p$-Brauer characters and $p$-regular units $u$, where the sum is taken only over the $p$-regular conjugacy classes (an element is called \emph{$p$-regular} if its order is not divisible by $p$). Assuming one knows inductively the character values of $u^d$ for all $d \mid k$ with $d \not= 1$, one has for $\psi$, an ordinary or a Brauer character in characteristic $p \nmid k$, and every $\ell$ a condition on the $\varepsilon_x(u)$ as follows:
  \begin{equation}\label{eq:HeLP-constraints}
    \sum_{x^G} \frac{\operatorname{Tr}_{\QQ(\zeta)/\QQ}(\psi(x)\zeta^{-\ell})}{k} \varepsilon_x(u) + a_\ell(u, \psi) \in \ZZ_{\geq 0},
 \end{equation}
 where the $a_\ell(u, \psi) = \frac{1}{k}\sum_{1 \not= d \mid k} \operatorname{Tr}_{\QQ(\zeta^d)/\QQ}(\psi(u^d)\zeta^{-d \ell})$ are assumed to be 'known'. By \cref{prop:pas_0} it is enough to take the sum \eqref{eq:HeLP-constraints} over classes of elements having an order dividing $k$.
 
\begin{remark}
 
\begin{itemize} 
 \item The bounds for the partial augmentations obtained in \cite[Corollary 2.3]{HalesLutharPassi} are encoded in the character table, so they will not add new information to the algorithm.
 \item It is intrinsic in the formula that the $\mu_\ell$'s sum up to the degree of the character.
 \item Note that for any set of partial augmentations fulfilling all the constraints obtained above for all irreducible ordinary characters of $G$ one can explicitly write down a unit in $K G$ having these partial augmentations, for $K$ a splitting field of $G$.
\end{itemize}
\end{remark}

\section{The extended Wagner Test}\label{section:wagner-test}

The program uses also a criterion proved in a special form by Roland Wagner in his Diplomarbeit in 1995. The more general case is given in \cite[Remark 6]{BovdiHertweck} and follows from a well known Lemma, recorded e.g. in \cite[Lemma (7.1)]{SehgalBook2}. For sake of completeness we include a proof. We write $g \sim h$, if $g$ and $h$ are conjugate in a group $G.$

\begin{proposition} Let $G$ be a finite group, $s$ some element in $G$ and $u \in \V(\ZZ G)$, $o(u) = p^j m$ with $p$ a prime and $m \not= 1$. Then $\smashoperator[r]{\sum\limits_{x^G,\ x^{p^j} \sim s}} \varepsilon_x(u) \equiv \varepsilon_{s}(u^{p^j}) \mod p$. \end{proposition}

\begin{proof} Let $u = \sum_{g\in G} u_g g \in\V(\ZZ G),$ set $q = p^j$ and $v = u^q$. Then by definition of the product in the group ring: 
\begin{equation} \varepsilon_s(v) = \sum_{\substack{(g_1, ..., g_q) \in G^q \\ g_1  ...  g_q \sim s}} \prod_{j=1}^q u_{g_j}. \label{vsum} \end{equation}
The set over which the sum is taken can be decomposed into $\mathcal{M} = \{(g, ..., g) \in G^q: g^q \sim s \}$ and $\mathcal{N} = \{ {(g_1, ..., g_q) \in G^q}:  g_1  ...  g_q \sim s \; \text{and} \; \exists \, r, r' : g_r \not= g_{r'}\}$. 

The cyclic group $C_q = \langle t \rangle$ of order $q$ acts on the set $\mathcal{N}$ by letting the generator $t$ shift the entries of a tuple to the left, i.e. $(g_1, g_2, g_3, ..., g_q) \cdot t = (g_2, g_3, ..., g_q, g_1)$. Note that all orbits have length $p^i$ with $i \geq 1$.  For elements in the same orbit, the same integer is summed up in \eqref{vsum}.
Hence using Fermat's little Theorem:
\[\varepsilon_s(v) =  \smashoperator[r]{\sum_{(g, ..., g) \in \mathcal{M}}} u_g^q + \sum_{(g_1, ..., g_q) \in \mathcal{N}} \prod_{j=1}^q u_{g_j} \equiv \smashoperator[r]{\sum_{(g, ..., g) \in \mathcal{M}}} u_g^q \equiv \smashoperator[r]{\sum_{(g, ..., g) \in \mathcal{M}}} u_g \equiv \smashoperator[r]{\sum_{x^G,\ x^{p^j} \sim s}} \varepsilon_x(u) \mod p. \qedhere \] 
\end{proof}

By induction and the Berman-Higman result (\cref{prop:pas_0} \ref{prop:BH} Wagner obtained the following. For units of prime power order the result also follows from \cite[Theorem 4.1]{CL}.
\begin{corollary}[Wagner]
 Let $G$ be a finite group, $u \in \V(\ZZ G)$, $o(u) = p^j m$ with $p$ a prime and $m \not= 1$. Then $\smashoperator[r]{\sum\limits_{x^G,\ o(x) = p^j}} \varepsilon_x(u) \equiv 0 \mod p$. \end{corollary}

\begin{example}
Let $G$ be the Mathieu group of degree 11. There exists only one conjugacy class of involutions in $G$, call it $\cc{2a}$. After applying HeLP (i.e.\ \cref{prop:HeLP-constraints}) for a unit $u$ of order 12 in $\V(\ZZ G)$ one obtains two possible partial augmentations for $u$. One of these possibilities satisfies $\varepsilon_{\cc{2a}}(u) = 1$ while the other satisfies $\varepsilon_{\cc{2a}}(u) = -1$  \cite{KonovalovM11}. Both possibilities however do not fulfill the constraints of Wagner's result and thus there are no torsion units of order 12 in $\V(\ZZ G)$ and the order of any torsion unit in $\V(\ZZ G)$ coincides with the order of an element in $G.$ 
\end{example}

\section{Implementation}\label{section:implementation}

\subsection{Further results}
We used several results in our implementation, which are not consequences of the HeLP-method and which we list here. The first one is a direct consequence of the Fong-Swan-Rukolaine Theorem \cite[Theorem 22.1]{CR1}.

\begin{proposition}
Let $G$ be a $p$-solvable group and $u \in \V(\ZZ G)$ a torsion unit of order prime to $p$. Then the restrictions on the possible partial augmentations of $u$ one can obtain using the $p$-Brauer table of $G$ are the same as when using the ordinary character table of $G.$
\end{proposition}

Often when applying the HeLP-method the inequalities one has to solve get very messy even for groups where e.g.\ the Zassenhaus Conjecture is known for a long time. As an example one might think of cyclic groups. For that reason the functions testing (ZC) and (PQ) in our implementation also use the following results instead of solving any inequalities in these situations:
\begin{proposition}\label{Weiss}
\begin{enumerate}[label=\alph*)]
\item[a)] (ZC) holds for nilpotent groups. \hfill Weiss \cite{Weiss91}
\item[b)] (PQ) has an affirmative answer for solvable groups. \hfill Kimmerle \cite{KimmiPQ}
\end{enumerate}
\end{proposition}

\begin{remark}
The above results do of cause not cover all knowledge on (ZC) and (PQ). E.g. (ZC) is known for cyclic-by-abelian groups \cite{CyclicByAbelian}, while (PQ) is known for the groups $\operatorname{PSL}(2,p)$ where $p$ denotes a prime \cite{HertweckBrauer}. However we decided for simplicity only to use the results above in the first version of the package.
\end{remark}

\subsection{Main functions of the HeLP-package}

The function \texttt{HeLP\_ZC} checks whether (ZC) can be verified using the character tables and Brauer tables available in \textsf{GAP}. For an element $u$ of order $k$ and $p$ and $q$ different prime divisors of $k$ we call partial augmentations of $u^p$ and $u^q$ \emph{compatible}, if $(u^p)^q$ and $(u^q)^p$ have the same partial augmentations.

\begin{algorithm}[H]%
 \KwData{group or ordinary character table of a group}
 \KwResult{\texttt{true} or \texttt{false}}
 \If{$G$ nilpotent}{\texttt{return true} (See \cref{Weiss})}
 \eIf{$G$ solvable}{\texttt{OrdersToCheck :=} Orders of elements in $G$ (See \cref{prop:orders})}{\texttt{OrdersToCheck :=} Divisors of $\exp G$} 
 \For{$k = o(u)$ in \texttt{OrdersToCheck}}{
   \For{all prime divisors $p$ of $o(u)$ and all possible partial augmentations of $u^p$}{
   \If{partial augmentations are compatible}{
     Construct the HeLP-systems for all relevant character tables and find its solutions
    }
  }
  Apply the Wagner test for order $k$\\
  Save the resulting possibilities for partial augmentations of units of order $k$ in the global variable \texttt{HeLP\_sol}
 }
 \texttt{return} only 'trivial' partial augmentations are admissible
 \caption{HeLP\_ZC}
\end{algorithm}

The function \texttt{HeLP\_PQ} checks whether (PQ) can be verified using the character tables and Brauer tables available in \textsf{GAP}. It works in a similar way as \texttt{HeLP\_ZC} but only checks the orders relevant for the Prime Graph Question.

There are many more functions available in the package implementing the HeLP-method for more specific checks, e.g.\ allowing to check a specific order of torsion units, cf.\ the reference manual of the package.

\subsection{Non-standard characters} Unfortunately not all character and Brauer tables known are yet readily available in \textsf{GAP} and the package can not be immediately applied in that case. However there are some workarounds.

\begin{example} The Brauer table modulo $7$ of $\operatorname{PSL}(2,49)$ is known generically, but not yet included in the \textsf{GAP} Character Table Library \cite{CTblLib}. However our implementation allows to use any class function of a group which may be entered e.g.\ manually. In this way also induced characters or other kind of class functions may be used and it is, among other things, possible to prove (ZC) for $\operatorname{PSL}(2,49)$. \end{example}

\begin{example}
Let $G$ be the projective unitary group $\operatorname{PSU}(3,8)$ and $A$ its automorphism group. Assume the goal is to check (PQ) for $A$. To obtain that one needs to exclude the existence of units of order $2\cdot 19$ and $7\cdot 19$ in $\V(\ZZ A).$ The character table of $G$ is available in \textsf{GAP} while the one of $A$ is not. However inducing the second and third irreducible character of $G$ to characters of $A$ one obtains two characters of $A$. The HeLP-constrains following from these two characters are already strong enough to prove (PQ) for $A$. 
\end{example}

\subsection{Solving the inequalities} Applying the HeLP-method involves solving the integral linear inequalities described after \cref{prop:HeLP-constraints}. This is a hard task in general and though always doable in theory this may take a lot of time when there are many inequalities and variables involved. A good solver of such systems is the main ingredient here. Our implementation allows to use two different solvers: The software system 4ti2 \cite{4ti2} and/or the system Normaliz \cite{Normaliz, NormalizArticle}. We chose those solvers for two reasons: They are good solvers and there exist \textsf{GAP}-Interfaces for them.  We could also include an other solvers as soon as a \textsf{GAP}-Interface will be available for it. To reduce the size of the system that has to be solved the package uses 'redund' from the lrslib software \cite{lrslib}. When using 4ti2, in many cases this leads to a remarkable speed up, sometimes however this also slows the calculations down, so that there is an option implemented to switch the use of 'redund' on and off.

\subsection{$p$-constant characters} If one is interested especially in solving (PQ) there is often a way to reduce the system one has to solve introduced by V. Bovdi and A. Konovalov in \cite{BKHS}. Assume one is studying the possible partial augmentations of units of order $p\cdot q$, where $p$ and $q$ are different primes. Let $\chi$ be a character which is constant on all conjugacy classes of elements of order $p$, a so called \emph{$p$-constant character}. Then the coefficients appearing in the HeLP-constrains provided by $\chi$ at partial augmentations of elements order $p$ are always the same. Thus one can reduce the number of variables involved by replacing all the partial augmentations of elements of order $p$ by their sum. This way one also does not need to know the partial augmentations of elements of order $p$ - their sum is 1 in any case. Often it suffices to study only $p$-constant characters to exclude the possibility of existence of units of order $p\cdot q$ and this functionality is also provided by the package.

\bibliographystyle{amsalpha}
\bibliography{help.bib}

\newcommand{\etalchar}[1]{$^{#1}$}
\providecommand{\bysame}{\leavevmode\hbox to3em{\hrulefill}\thinspace}
\providecommand{\MR}{\relax\ifhmode\unskip\space\fi MR }
\providecommand{\MRhref}[2]{%
  \href{http://www.ams.org/mathscinet-getitem?mr=#1}{#2}
}
\providecommand{\href}[2]{#2}
\begin{thebibliography}{MRSW87}

\bibitem[Avi]{lrslib}
D.~Avis, \emph{lrslib -- reverse search vertex enumeration program}, available
  at \url{http://cgm.cs.mcgill.ca/~avis/C/lrs.html}.

\bibitem[BH08]{BovdiHertweck}
Victor Bovdi and Martin Hertweck, \emph{Zassenhaus conjecture for central
  extensions of {$S_5$}}, J. Group Theory \textbf{11} (2008), no.~1, 63--74.

\bibitem[BIR{\etalchar{+}}]{Normaliz}
W.~Bruns, B.~Ichim, T.~R\"omer, R.~Sieg, and C.~S\"oger, \emph{Normaliz.
  algorithms for rational cones and affine monoids}, Available at
  \url{http://normaliz.uos.de}.

\bibitem[BIS16]{NormalizArticle}
W.~Bruns, B.~Ichim, and C.~S{\"o}ger, \emph{The power of pyramid decomposition
  in {N}ormaliz}, J. Symbolic Comput. \textbf{74} (2016), 513--536.

\bibitem[BK07]{KonovalovM11}
V.A. Bovdi and A.B. Konovalov, \emph{Integral group ring of the first {M}athieu
  simple group}, Groups {S}t. {A}ndrews 2005. {V}ol. 1, London Math. Soc.
  Lecture Note Ser., vol. 339, Cambridge Univ. Press, Cambridge, 2007,
  pp.~237--245.

\bibitem[BK10]{BKHS}
V.~A. Bovdi and A.~B. Konovalov, \emph{Torsion units in integral group ring of
  {H}igman-{S}ims simple group}, Studia Sci. Math. Hungar. \textbf{47} (2010),
  no.~1, 1--11.

\bibitem[BM16]{HeLP}
A.~B{\"a}chle and L.~Margolis, \emph{{HeLP} -- {H}ertweck-{L}uthar-{P}assi
  method}, \url{http://homepages.vub.ac.be/abachle/help/}, 2016, \textsf{GAP}
  package, {V}ersion 3.0.

\bibitem[Bre12]{CTblLib}
T.~Breuer, \emph{The \textsf{GAP} {C}haracter {T}able {L}ibrary},
  \url{http://www.math.rwth-aachen.de/\~Thomas.Breuer/ctbllib}, May 2012,
  \textsf{GAP} package, {V}ersion 1.2.1.

\bibitem[CL65]{CL}
J.A. Cohn and D.~Livingstone, \emph{On the structure of group algebras {I}},
  Canadian Journal of Mathematics \textbf{17} (1965), 583--593.

\bibitem[CMdR13]{CyclicByAbelian}
M.~Caicedo, L.~Margolis, and {\'A}.~del R{\'{\i}}o, \emph{Zassenhaus conjecture
  for cyclic-by-abelian groups}, J. Lond. Math. Soc. (2) \textbf{88} (2013),
  no.~1, 65--78.

\bibitem[CR90]{CR1}
C.W. Curtis and I.~Reiner, \emph{Methods of representation theory. {V}ol. {I}},
  Wiley Classics Library, John Wiley \& Sons, Inc., New York, 1990, With
  applications to finite groups and orders, Reprint of the 1981 original, A
  Wiley-Interscience Publication.

\bibitem[Her07]{HertweckBrauer}
M.~Hertweck, \emph{Partial {A}ugmentations and {B}rauer character values of
  torsion units in group rings},
  \href{http://arxiv.org/abs/math/0612429v2}{\nolinkurl{arXiv:math.RA/0612429v2
  [math.RA]}} (2007).

\bibitem[Her08]{HertweckSolvable}
M.~Hertweck, \emph{The orders of torsion units in integral group rings of
  finite solvable groups}, Comm. Algebra \textbf{36} (2008), no.~10,
  3585--3588.

\bibitem[Hig40]{HigmanThesis}
G.~Higman, \emph{Units in group rings}, D. phil. thesis, Oxford Univ., 1940.

\bibitem[HLP90]{HalesLutharPassi}
A.W. Hales, I.S. Luthar, and I~B.S. Passi, \emph{Partial augmentations and
  {J}ordan decomposition in group rings}, Comm. Algebra \textbf{18} (1990),
  no.~7, 2327--2341.

\bibitem[Kim06]{KimmiPQ}
W.~Kimmerle, \emph{On the prime graph of the unit group of integral group rings
  of finite groups}, Groups, rings and algebras, Contemp. Math., vol. 420,
  Amer. Math. Soc., Providence, RI, 2006, pp.~215--228.

\bibitem[LP89]{LP89}
I.S. Luthar and I.B.S. Passi, \emph{Zassenhaus conjecture for {$A_5$}}, Proc.
  Indian Acad. Sci. Math. Sci. \textbf{99} (1989), no.~1, 1--5.

\bibitem[MRSW87]{MRSW}
Z.~Marciniak, J.~Ritter, S.~K. Sehgal, and A.~Weiss, \emph{Torsion units in
  integral group rings of some metabelian groups. {II}}, J. Number Theory
  \textbf{25} (1987), no.~3, 340--352.

\bibitem[Seh93]{SehgalBook2}
S.K. Sehgal, \emph{Units in integral group rings}, Pitman Monographs and
  Surveys in Pure and Applied Mathematics, vol.~69, Longman Scientific \&
  Technical, Harlow, 1993.

\bibitem[tt]{4ti2}
4ti2 team, \emph{4ti2---a software package for algebraic, geometric and
  combinatorial problems on linear spaces}, {A}vailable at \url{www.4ti2.de}.

\bibitem[Wei88]{Weiss88}
A.~Weiss, \emph{Rigidity of {$p$}-adic {$p$}-torsion}, Ann. of Math. (2)
  \textbf{127} (1988), no.~2, 317--332.

\bibitem[Wei91]{Weiss91}
\bysame, \emph{Torsion units in integral group rings}, J. Reine Angew. Math.
  \textbf{415} (1991), 175--187.

\bibitem[Zas74]{Zassenhaus}
H.~J. Zassenhaus, \emph{On the torsion units of finite group rings}, Studies in
  mathematics (in honor of {A}. {A}lmeida {C}osta), Instituto de Alta Cultura,
  Lisbon, 1974, pp.~119--126.

\end{thebibliography}

\vspace{.25cm}
\begin{center}  The package is available at \url{http://homepages.vub.ac.be/abachle/help/}.\end{center}
 \vspace{.25cm}
 
\end{document}